\documentclass[12pt,reqno]{amsart}

\usepackage{etex}
\usepackage[utf8]{inputenc}
\usepackage{amsfonts}
\usepackage{amsmath}
\usepackage{amssymb}
\usepackage{amsthm}
\usepackage{mathrsfs}
\usepackage{stmaryrd}
\usepackage{color}
\usepackage[english]{babel}
\usepackage[T1]{fontenc}
\usepackage{url}
\usepackage{graphicx}
\usepackage{hyperref}
\usepackage{caption}
\usepackage{enumerate}
\usepackage{epstopdf}
\usepackage{hyphenat}
\usepackage{float}
\usepackage{indentfirst}
\usepackage[export]{adjustbox}
\usepackage{tikz}
\usepackage{color}
\usepackage[all]{xy}
\usetikzlibrary{matrix,arrows,decorations.pathmorphing}
\usepackage{booktabs}
\usepackage{array}
\newcolumntype{P}[1]{>{\centering\arraybackslash}p{#1}}

\allowdisplaybreaks

\newtheorem{theorem}{Theorem}[section]
\newtheorem{lemma}[theorem]{Lemma}
\newtheorem{proposition}[theorem]{Proposition}
\newtheorem{corollary}[theorem]{Corollary}
\newtheorem{theoremx}{Theorem}

\newtheorem*{theoremb}{Theorem B}

\theoremstyle{definition}

\newtheorem{remark}[theorem]{Remark}
\newtheorem{notation}[theorem]{Notation}
\newtheorem{example}[theorem]{Example}
\newtheorem{definition}[theorem]{Definition}
\newtheorem{question}[theorem]{Question}

\DeclareMathOperator{\PGL}{PGL}
\DeclareMathOperator{\sign}{sign}
\DeclareMathOperator{\chara}{char}

\makeatletter
\def\l@subsection{\@tocline{2}{0pt}{2.5pc}{5pc}{}}
\makeatother

\author{Alessio Caminata}
\address{{\small Dipartimento di Matematica, Universit\`a di Genova, via Dodecaneso 35, 16146, Genova, Italy}}
\email{caminata@dima.unige.it}

\author{Luca Schaffler}
\address{{\small Department of Mathematics, KTH Royal Institute of Technology, SE-100 44 Stockholm, Sweden}}
\email{lucsch@math.kth.se}

\title{A Pascal's Theorem for rational normal curves}

\begin{document}

\thanks{\textit{Mathematics Subject Classification (2020)}: 14A25, 14H50, 51N35. 
	\\ \indent \textit{Keywords and phrases:} Pascal's Theorem, rational normal curve, twisted cubic, Grassmann--Cayley algebra, bracket ring}

\maketitle

\begin{abstract}
Pascal's Theorem gives a synthetic geometric condition for six points $a,\ldots,f$ in $\mathbb{P}^2$ to lie on a conic. Namely, that the intersection points $\overline{ab}\cap\overline{de}$, $\overline{af}\cap\overline{dc}$, $\overline{ef}\cap\overline{bc}$ are aligned. One could ask an analogous question in higher dimension: is there a coordinate-free condition for $d+4$ points in $\mathbb{P}^d$ to lie on a degree $d$ rational normal curve? In this paper we find many of these conditions by writing in the Grassmann--Cayley algebra the defining equations of the parameter space of $d+4$ ordered points in $\mathbb{P}^d$ that lie on a rational normal curve. These equations were introduced and studied in a previous joint work of the authors with Giansiracusa and Moon. We conclude with an application in the case of seven points on a twisted cubic.
\end{abstract}


\section{Introduction}

Pascal's Theorem is a classic result in plane projective geometry.
It says that if six points $a,\dots,f$ in $\mathbb{P}^2$ lie on a conic then the three intersection points $\overline{ab}\cap\overline{de}$, $\overline{af}\cap\overline{dc}$, $\overline{ef}\cap\overline{bc}$ are aligned \cite{Pas40}.
Actually, this is a generalization of an even older result of Pappus, for which the same conclusion holds if instead of a conic we require three points to lie on a line and the other three on another line. 
Pappus's Theorem can be seen as the special case of Pascal's with a degenerate conic of two lines.
Pappus--Pascal Theorem is also known as the Mystic Hexagon Theorem, with reference to the hexagon with vertices the six points.

\par The converse of this result is also true and is due to Braikenridge and Maclaurin.
It states that if the three intersection points of the three pairs of lines through opposite sides of a hexagon lie on a line, then the six vertices of the hexagon lie on a (possibly degenerate) conic \cite{Bra33, Mac35}.
In the sequel, we will refer to these statements simply as Pascal's Theorem, meaning that the two implications hold, and the conic might be degenerate.

\par One of the strengths of Pascal's Theorem is that it converts a quadratic condition, the fact that six points lie on a conic, to a linear condition, namely asking for three points to be aligned.
Therefore it is natural to ask whether such a result could be generalized. In fact, many generalizations appear in the literature, and we will soon go back to them. 
Now, let us state the question  which is the main object of investigation of this paper.
\begin{question}\label{mainquestion}
	Is there a synthetic linear condition for $d+4$ points in $\mathbb{P}^d$ to lie on a degree $d$ rational normal curve?
\end{question}

We recall that a \emph{rational normal curve} in $\mathbb{P}^d$ is a smooth rational curve of degree $d$.
By Castelnuovo's Lemma, there is always a rational normal curve passing through $d+3$ points in general linear position in projective space $\mathbb{P}^d$.
For example, for $d=2$ we have that $5$ points always lie on a conic, and Pascal's Theorem gives a synthetic linear condition for $d+4=6$ points to lie on a conic.
For $d\geq3$, we are able to provide an answer to Question~\ref{mainquestion}, in the following form. We work over an algebraically closed field $\Bbbk$ of arbitrary characteristic, unless otherwise specified.

\begin{theoremx}[see Corollary \ref{corollary-main}]
\label{maintheoreminintro}
Let $P_1,\ldots,P_{d+4}\in\mathbb{P}^d$ be points in general linear position. Then $P_1,\ldots,P_{d+4}$ lie on a rational normal curve if and only if for every $I=\{i_1<\ldots<i_6\}\subseteq\{1,\dots,d+4\}$, $I^c=\{j_1<\ldots<j_{d-2}\}$, the following $d+1$ points lie on a hyperplane:
\begin{itemize}
\item The intersection of the line $P_{i_1}P_{i_2}$ with the hyperplane $P_{i_4}P_{i_5}P_{j_1}\ldots P_{j_{d-2}}$;
\item The intersection of the line $P_{i_2}P_{i_3}$ with the hyperplane $P_{i_5}P_{i_6}P_{j_1}\ldots P_{j_{d-2}}$;
\item The intersection of the line $P_{i_3}P_{i_4}$ with the hyperplane $P_{i_1}P_{i_6}P_{j_1}\ldots P_{j_{d-2}}$;
\item The points $P_{j_1},\ldots,P_{j_{d-2}}$.
\end{itemize}
\end{theoremx}

As Pascal's Theorem is true also for degenerate conics, i.e. two lines, a more general form of Theorem \ref{maintheoreminintro} holds for appropriate degenerations of rational normal curves, the \emph{quasi-Veronese curves}. We refer to \S \ref{equationsforpointsonarationalnormalcurve} for the definition and examples of quasi-Veronese curves, and to Corollary \ref{corollary-main} for the precise statement of the above result.


\subsection{Methods employed}
Our main tool for the proof of Theorem \ref{maintheoreminintro} is the Grassmann--Cayley algebra.
The Grassmann--Cayley algebra of a given finite dimensional $\Bbbk$-vector space, is nothing else than the exterior algebra of the vector space together with two operations: the join denoted by $\vee$, which is just the standard wedge product, and the meet, which is denoted by $\wedge$.
The reason for this apparently strange change of notations is geometric. In fact, equations in the Grassmann--Cayley algebra of a vector space $V$ can be used to represent linear dependence among linear subspaces of the projective space $\mathbb{P}(V)$, where the join corresponds to the sum of linear spaces and the meet to the intersection.
For instance, the collinearity of the three points $\overline{ab}\cap\overline{de}$, $\overline{af}\cap\overline{dc}$, $\overline{ef}\cap\overline{bc}$ in Pascal's Theorem, can be rewritten in the Grassmann--Cayley algebra as follows:
\[
((a\vee b)\wedge(d\vee e))\vee((a\vee f)\wedge(d\vee c))\vee((e\vee f)\wedge(b\vee c))=0.
\]
By introducing coordinates $a=(a_1,a_2,a_3)$ etc. for each point, one can expand the previous expression to obtain a multihomogeneous equation in the coordinates of the points in $\mathbb{P}^2$.
By using appropriate syzygies, this can be written as the following algebraic combination of determinants in the form $[abc]=\det(abc)$, etc:
\begin{equation}\label{eq:V26intro} 
[abc][ade][bdf][cef]-[abd][ace][bcf][def]=0.
\end{equation}
It is classical and well-known that this equation is equivalent to requiring that $a,\dots,f$ lie on a (possibly degenerate) conic (cf.,~(\cite[p.118]{Cob61} and \cite[Example~3.4.3]{Stu08}).
Thus, one obtains a Grassmann--Cayley algebra proof of Pascal's Theorem.

\par We would like to mimic the same story in higher dimension.
Let $d\geq2$ be an integer, and consider the parameter space of $d+4$ points in $\mathbb{P}^d$ supported on a rational normal curve. 
More precisely, we define the variety $V_{d,d+4}\subseteq(\mathbb{P}^d)^{d+4}$ as the Zariski closure of the subset of $(d+4)$-tuples of points in $\mathbb{P}^d$ that lie on a rational normal curve.
For example, $V_{2,6}$ is simply the hypersurface in $(\mathbb{P}^2)^6$ defined by equation~\eqref{eq:V26intro}.
More generally, in a previous joint work with Giansiracusa and Moon \cite{CGMS18}, we were able to provide multihomogeneous equations that cut out set-theoretically $V_{d,d+4}$ union with the the locus of  $(d+4)$-point configurations in $\mathbb{P}^d$ supported on a hyperplane (see \S\ref{equationsforpointsonarationalnormalcurve} for precise definitions and more details).
As for the two-dimensional situation, these equations can be written as algebraic combinations of determinants.
So one may try to convert them into Grassmann--Cayley algebra expressions in order to obtain a synthetic geometric statement in the spirit of Pascal's Theorem.
Unfortunately, while passing from Grassmann--Cayley algebra expressions to multihomogeneous equations is always possible and requires only tedious, but straightforward computations, the other direction is in general highly non-trivial, and not even always possible.
Determining whether a given expression can be written in the Grassmann--Cayley algebra, and, if possible, determining such expression, is called the \emph{Cayley Factorization Problem} (\cite{SW91}, \cite{Whi91}, \cite[\S3.5]{Stu08}, \cite[\S4.5]{ST19}). This is a hard problem, and no general algorithm is known.
We remark that an important partial result is given by N. White \cite{Whi91}, who provides an algorithm for the multilinear case (i.e. each point occurs exactly once in the monomials).
However, the equations we have in our case are not multilinear, therefore we cannot take advantage of White's algorithm.

\par To solve this problem, we introduce a technique to lift syzygies from the two-dimensional to the $d$-dimensional situation (see \S\ref{liftvanderwaerdensyzygies} for details).
Using this technique, we are able to rewrite the equations for $V_{d,d+4}$ in the Grassmann--Cayley algebra, obtaining the coordinate-free description claimed in Theorem~\ref{maintheoreminintro}. More precisely, we prove the following.

\begin{theoremx}\label{maintheoremintro2}
	Let $d\geq3$, let $P_1,\ldots,P_{d+4}$ be points in $\mathbb{P}^d$ not on a hyperplane. Then the following are equivalent:
	\begin{enumerate}[(i)]
		\item $(P_1,\ldots,P_{d+4})\in V_{d,d+4}$ (equivalently, they lie on a quasi-Veronese curve);
		\item For every $I=\{i_1<\ldots<i_6\}\subseteq\{1,\dots,d+4\}$, $I^c=\{j_1<\ldots<j_{d-2}\}$ the following equality in the Grassmann--Cayley algebra holds:
		\begin{gather*}
		(P_{i_1}P_{i_2}\wedge P_{i_4}P_{i_5}P_{j_1}\ldots P_{j_{d-2}})\vee(P_{i_2}P_{i_3}\wedge P_{i_5}P_{i_6}P_{j_1}\ldots P_{j_{d-2}})\\
		\vee(P_{i_3}P_{i_4}\wedge P_{i_6}P_{i_1}P_{j_1}\ldots P_{j_{d-2}})\vee P_{j_1}\ldots P_{j_{d-2}}=0.
		\end{gather*}
	\end{enumerate}
\end{theoremx}

From the geometric interpretation of the Grassmann--Cayley algebra expression in Theorem~\ref{maintheoremintro2}, one obtains immediately Theorem~\ref{maintheoreminintro}, which is the claimed generalization of Pascal's Theorem. We remark that in Theorem~\ref{equivalentformulations} we find many equivalent ways of rewriting the Grassmann--Cayley expression in Theorem~\ref{maintheoremintro2} (ii), and these provide different, but equivalent, reformulations of Theorem~\ref{maintheoreminintro}. Finally, by dualizing one of the implications of Theorem~\ref{maintheoreminintro}, we obtain a generalization of Brianchon's Theorem to rational normal curves (see Corollary~\ref{generalizedbrianchon'stheorem}).


\subsection{Historical context}
In the literature, Pascal's Theorem was generalized in many different directions. This gave rise to a great abundance of results in projective geometry, which we briefly survey.

In \cite{Mob48} M\"obius proved the following. Assume a polygon with $4n+2$ sides is inscribed in an irreducible conic. Determine $2n+1$ points by extending opposite edges until they meet. If $2n$ of these $2n+1$ points of intersection lie on a line, then the last point also lies on the line. The case $n=1$ recovers Pascal's Theorem. The classical theorem of Chasles \cite{Cha85} (stating that if we have two planar cubics meeting at nine points and a third cubic passes through eight of the nine points, then the third cubic also passes through the ninth) implies Pascal's Theorem if we consider reducible cubics. Chasles's Theorem was generalized by Cayley \cite{Cay43} and Bacharach \cite{Bac86} to planar curves of arbitrary degrees. See \cite{EGH96} for a detailed survey about these results and further developments in this direction.

In \cite{Jam30}, James fixes five of the six points on a conic in Pascal's Theorem, and allows the sixth one to move away from the conic. The object of investigation is to determine the loci of the varying point when certain restrictions have been placed upon the triangle formed by the intersections of opposite sides of the hexagon. Beniamino Segre proved results about lines in $\mathbb{P}^3$ and $\mathbb{P}^4$ in the spirit of Pascal's Theorem. For instance, \cite{Seg45} gives a necessary and sufficient condition for a double-four in $\mathbb{P}^3$ to lie on a cubic surface, and this boils down to the linear dependence of certain point configurations on the lines (a double-four consists of two sets of four skew lines $a_1,\ldots,a_4$ and $b_1,\ldots,b_4$ such that $a_i$ and $b_i$ are skew and $a_i\cap b_j\neq\emptyset$ for all $i\neq j$).

More recently, Borodzik and \.Zo\l\k adek generalized Pascal's Theorem to the case of a general planar cubic and for rational planar cubics. For the precise statement we refer to \cite[Theorem 4.4 and Theorem 5.1]{BZ02}. Finally,  a simplified version of the main result of \cite{Tra13} says that if two sets of $k$ lines meet in $k^2$ distinct points, and if $dk$ of those points lie on an irreducible curve of degree $d$, then the remaining $k(k-d)$ points lie on a unique curve of degree $k-d$ (the case $k=3$ and $d=2$ recovers Pascal's Theorem).

As it appears from the above discussion, we were not able to find in the literature a generalization of Pascal's Theorem considering higher degree rational normal curves, which instead is the case of interest in the current paper. 


\subsection{Organization of the paper}

We now outline the structure of the paper. In Section \ref{preliminaries} we collect some preliminary results divided into two parts. The first part (\S\ref{thebracketring}, \S\ref{grassmann-cayleyalgebra}) briefly reviews definitions and main results about the Grassmann--Cayley algebra and its geometric interpretation. In the second part (\S\ref{equationsforpointsonarationalnormalcurve}) we consider the parameter space $V_{d,d+4}$ and its defining equations.
Sections \ref{liftvanderwaerdensyzygies} and \ref{alternativedescriptionequationswithlifts} are of technical nature: in \S\ref{liftvanderwaerdensyzygies} we introduce a technique to lift van der Waerden syzygies of multihomogeneous polynomials from the plane situation to higher dimension (Definition~\ref{def:lift}), and in \S\ref{alternativedescriptionequationswithlifts} we rewrite the equations of $V_{d,d+4}$ in a way that is compatible with these lifts. These results are then used in the proof of the main theorem, which is contained in \S\ref{sectionwithmainresult}. 
Finally, in \S\ref{applicationofresults} we combine our result with a 100-years-old theorem of H. White \cite{Whi15} to study the geometry of seven points on a twisted cubic.

\section*{acknowledgements}
We would like to thank Noah Giansiracusa, Han-Bom Moon, Jessica Sidman, and Will Traves for their interest in this project. In particular, we are grateful to Jessica Sidman for raising the question that is object of this paper. We also thank Alessandro Oneto for his suggestions. Finally, we thank the anonymous referee for the valuable comments and feedback. Figure~\ref{picture7pointswithtwistedcubic} was realized using the software GeoGebra, Copyright \copyright~International GeoGebra Institute, 2013. The first author was supported by the European Union's Horizon~2020 research and innovation programme under grant agreement No.~701807.


\section{Preliminaries}
\label{preliminaries}
In this section, we present preliminary material that we rely on in the rest of the paper. For the reader's convenience, in \S\ref{thebracketring} and \S\ref{grassmann-cayleyalgebra} we briefly survey the main definitions and classic facts on the bracket ring and Grassmann--Cayley algebra. A more detailed exposition and proofs can be found in \cite[Chapter 3]{Stu08} and \cite{BBR85,SW89}. See also the more recent \cite{ST19}.
Then, in \S\ref{equationsforpointsonarationalnormalcurve} we recall the main results on the equations of the variety $V_{d,d+4}$ from \cite{CGMS18} (for further applications of these equations in tropical geometry see \cite{CGMS20}).


\subsection{The bracket ring}
\label{thebracketring}

Let $\Bbbk$ be an algebraically closed field. Let $J\subseteq\mathbb{N}$ be a finite index set, and let $d\leq|J|$ be a positive integer. A \emph{bracket} is a formal expression $[\lambda_1\ldots\lambda_d]$ where $\lambda_1<\ldots<\lambda_d$, $\lambda_1,\dots,\lambda_d\in J$. Denote by $\Lambda(J,d)$ the set of all such brackets. If $J=\{1,\dots,n\}$, then we simply denote this set by $\Lambda(n,d)$. Define $\Bbbk[\Lambda(J,d)]$ to be the polynomial ring generated by the elements in $\Lambda(J,d)$. If $\pi\in S_d$ is any permutation, it is useful to define $[\lambda_{\pi(1)}\ldots\lambda_{\pi(d)}]=\sign(\pi)[\lambda_1\ldots\lambda_d]$. Moreover, if $\lambda_i=\lambda_j$ for some $i\neq j$, we set $[\lambda_1\ldots\lambda_d]=0$.

\par The \emph{generic coordinatization} is the algebra homomorphism
\[
\phi_{J,d}\colon\Bbbk[\Lambda(J,d)]\rightarrow\Bbbk[x_{ij}\mid 1\leq i\leq d, \ j\in J],
\]
defined by extending $[\lambda_1\ldots\lambda_d]\mapsto\det(x_{i\lambda_j})$. Denote by $\mathcal{I}_{J,d}$ the kernel of $\phi_{J,d}$. The elements of $\mathcal{I}_{J,d}$ are called \emph{syzygies} and the image of $\phi_{J,d}$ is called \emph{bracket ring}, which we denote by $\mathcal{B}_{J,d}$.
The generic coordinatization gives an identification $\mathcal{B}_{J,d}\cong\Bbbk[\Lambda(J,d)]/\mathcal{I}_{J,d}$. Therefore, by abuse of notation, we will often identify a formal bracket $[\lambda]$ with its associated determinant $\phi_{J,d}([\lambda])$.

\begin{definition}
Let $s\in\{1,\ldots,d\},\alpha\in\Lambda(J,s-1),\beta\in\Lambda(J,d+1)$, and $\gamma\in\Lambda(J,d-s)$. The \emph{van der Waerden} syzygy $[[\alpha\beta^\bullet\gamma]]$ is defined to be the following quadratic polynomial in $\Bbbk[\Lambda(J,d)]$:
\[
[[\alpha\beta^\bullet\gamma]]=\sum_{\tau\in\Lambda(d+1,s)}\sign(\tau,\tau^*)[\alpha\beta_{\tau_1^*}\ldots\beta_{\tau_{d+1-s}^*}][\beta_{\tau_1}\ldots\beta_{\tau_s}\gamma].
\]
\end{definition}
Let us clarify the notation introduced: for a bracket $\tau\in\Lambda(d+1,s)$, we let $\tau^*\in\Lambda(d+1,d+1-s)$ be the unique bracket consisting of the indices $\{1,\ldots,n\}\setminus\{\tau_1,\ldots,\tau_s\}$. By $\sign(\tau,\tau^*)$ we mean the sign of the permutation sending $\tau_1,\ldots,\tau_s$ to the first $s$ indices and $\tau_1^*,\ldots,\tau_{d+1-s}^*$ to the last $d+1-s$ indices. 
Computing the generic coordinatization of $[[\alpha\beta^\bullet\gamma]]$, one sees that $[[\alpha\beta^\bullet\gamma]]\in \mathcal{I}_{J,d}$.
Actually even more is true. In fact, a subset of the van der Waerden syzygies, the so-called \emph{straightening syzygies}, is a Gr\"obner basis of the ideal $\mathcal{I}_{J,d}$ with respect to a suitable term order (see \cite[Theorem 5.1]{SW89} and references there for previous related results). In particular, the van der Waerden syzygies generate $\mathcal{I}_{J,d}$.

\begin{example}
\label{firstexampleofvanderwaerdensyzygy}
We fix $d=3, J=\{1,\dots,6\}, s=2,\alpha=[1],\beta=[1346]$, and $\gamma=[5]$. Then, we obtain the van der Waerden syzygy
\begin{align*}
[[\alpha\beta^\bullet\gamma]]=[146][135]-[136][145]+[134][165]=[146][135]-[136][145]-[134][156].
\end{align*}
\end{example}

We conclude by recalling the following important result. Given a polynomial class $\overline{F}$ in the bracket ring $\mathcal{B}_{J,d}$, one would like to find a representative for $\overline{F}$ in ``standard form''. More precisely, consider brackets $[\lambda^1],\ldots,[\lambda^k]\in\Lambda(J,d)$. The monomial $[\lambda^1]\cdot\ldots\cdot[\lambda^k]\in\Bbbk[\Lambda(J,d)]$ is called \emph{standard} if $\lambda_i^1\leq\ldots\leq\lambda_i^k$ for all $i\in\{1,\ldots,d\}$. It turns out that the standard monomials in $\Bbbk[\Lambda(J,d)]$ form a $\Bbbk$-vector space basis for the bracket ring $\mathcal{B}_{J,d}$ (\cite[Theorem~3.1]{SW89}). Therefore, using appropriate syzygies, we can choose a representative for the class $\overline{F}\in\mathcal{B}_{J,d}$ whose monomials are in standard form. This is the so-called \emph{straightening algorithm}.


\subsection{The Grassmann--Cayley algebra}
\label{grassmann-cayleyalgebra}

Let $V$ be a $d$-dimensional $\Bbbk$-vector space. Given two vectors $v,w\in V$ the \emph{join} of $v$ and $w$, denoted by $v\vee w$, is the wedge of the two vectors in the exterior algebra $\Lambda(V)$ (this convention is adopted for geometric reasons). Often, to simplify our notation, we denote $v\vee w$ simply by $vw$. If $v_1,\ldots,v_k\in V$, then $v_1\ldots v_k$ is called an \emph{extensor of step $k$}. Let $e_1,\ldots,e_d$ be a fixed basis of $V$. If we identify $e_1\ldots e_d$ with $1$, then $v_1\ldots v_d$ equals the determinant of the matrix of coordinates with respect to the chosen basis. We denote such determinant by $[v_1\ldots v_d]$, which we still consider an extensor of step $d$.

\par Given two extensors $a_1\ldots a_j$ and $b_1\ldots b_k$ with $j+k\geq d$, we define their \emph{meet} as the following element of $\Lambda^{j+k-d}(V)$:
\begin{equation}\label{eq-meet}
(a_1\ldots a_j)\wedge(b_1\ldots b_k)=\sum_\sigma\sign(\sigma)[a_{\sigma(1)}\ldots a_{\sigma(d-k)}b_1\ldots b_k]a_{\sigma(d-k+1)}\ldots a_{\sigma(j)},
\end{equation}
where the sum is taken over all permutations $\sigma$ of $\{1,\ldots,j\}$ such that $\sigma(1)<\ldots<\sigma(d-k)$ and $\sigma(d-k+1)<\ldots<\sigma(j)$. The meet operation is associative and satisfies $A\wedge B=(-1)^{(d-j)(d-k)}B\wedge A$.

\par The \emph{Grassmann--Cayley algebra} is the vector space $\Lambda(V)$ together with the operations $\vee$ and $\wedge$ extended by distributivity. An expression in the Grassmann--Cayley algebra is called \emph{simple} if it is obtained by combining vectors in $V$ only using meet and join operations, not addition (see \cite[Example~3.3.3]{Stu08}). For instance, if $V$ is $3$-dimensional and $v_1,v_2,v_3\in V$, then $(v_1v_2)\wedge(v_1v_3)\wedge(v_2v_3)$ is a simple expression. Note that $(v_1v_2)\wedge(v_1v_3)\wedge(v_2v_3)$ is of step $0$, i.e. an element of $\Lambda^0(V)$.

\begin{remark} We point out that each simple expression of step $0$ in the Grassmann--Cayley algebra can be expanded giving an element of the bracket ring using \eqref{eq-meet}. On the other hand, not every homogeneous bracket polynomial can be obtained by expanding a Grassmann--Cayley algebra expression. Understanding whether this is possible is the so-called \emph{Cayley Factorization Problem} (\cite{SW91}, \cite{Whi91}, \cite[\S3.5]{Stu08}, \cite[\S4.5]{ST19}).
\end{remark}

\par The following argument gives a geometric interpretation of the elements of the Grassmann--Cayley algebra and the join and meet operations.
Let $A=a_1\ldots a_j$ be a non-zero extensor of step $j$. 
Let $\overline{A}$ be the be the $j$-dimensional vector subspace of $V$ generated by $a_1,\ldots,a_j$.
Observe that $\overline{A}$ is uniquely determined by $A$ and is independent of the representation chosen since $\overline{A}=\{v\in V: \ A\vee v=0\}$.
Conversely, each $j$-dimensional vector subspace $W\subseteq V$ uniquely determines, up to scalar multiplication, an extensor of step $j$: if $w_1,\ldots,w_j$ is a basis of $W$, then consider $w_1\ldots w_j$.
\par Keeping this interpretation in mind, the algebraic join of two extensors $A$ and $B$ corresponds to the linear span of the linear subspaces $\overline{A}$ and $\overline{B}$.
Similarly, the meet of $A$ and $B$ corresponds to the intersection of $\overline{A}$ and $\overline{B}$.
More precisely, we have the following result.

\begin{proposition}[Geometric interpretation, {\cite[Proposition~3.5 and Proposition~4.3]{BBR85}}]
\label{geometrictranslation}
Let $V$ be a $\Bbbk$-vector space of dimension $d$. Let $A=a_1\vee\ldots\vee a_j$ and $B=b_1\vee\ldots\vee b_k$ be two extensors of steps $j$ and $k$ respectively. Then
\begin{itemize}
		\item $A\vee B\neq0$ if and only if $a_1,\dots,a_j,b_1,\dots,b_k$ are linearly independent. In this case $\overline{A}+\overline{B}=\overline{A\vee B}=\mathrm{span}\{a_1,\dots,a_j,b_1,\dots,b_k\}$.
		\item Assume $j+k\geq d$. Then $A\wedge B\neq0$ if and only if $\overline{A}+\overline{B}=V$. In this case, $\overline{A}\cap\overline{B}=\overline{A\wedge B}$. In particular, $A\wedge B$ can be represented by an appropriate extensor.
\end{itemize}
\end{proposition}

\begin{example}[Pascal's Theorem]\label{example:pascal}
Let $P_1,\dots,P_6$ be six ordered points in $\mathbb{P}^2$. 
Using the geometric interpretation of Grassmann--Cayley algebra statements, the collinearity of the three points $\overline{P_1P_2}\cap \overline{P_4P_5}$, $\overline{P_2P_3}\cap \overline{P_5P_6}$, and  $\overline{P_3P_4}\cap\overline{P_6P_1}$ can be expressed as follows:
\[
(12\wedge 45)\vee(23\wedge 56)\vee(34\wedge 61)=0,
\]
 where to improve the readability, we denote each point $P_i$ just by its subscript $i$. Expanding the above expression in bracket polynomials yields
\begin{equation}\label{eq:pascal1}
[145][256][361][234]-[245][356][461][123]+[245][356][361][124]-[245][256][361][134]=0.
\end{equation}
Observe that these four bracket monomials are not standard, so we can straighten them using the following syzygies: 
\begin{equation}\label{eq:syzygiespascal}
\begin{split}
&[136][145]+[134][156]-[135][146]=0,\\
&[146][234]+[124][346]-[134][246]=0,\\
&[256][346]+[236][456]-[246][356]=0,\\
&[146][245]+[124][456]-[145][246]=0,\\
&[136][245]+[123][456]+[134][256]-[135][246]=0,\\
&[156][234]+[124][356]-[134][256]-[123][456]=0.
\end{split}
\end{equation}
Thus, we obtain the unique standard equation
\begin{equation}\label{eq:pascal2} 
[123][145][246][356]-[124][135][236][456]=0,
\end{equation}
which is equivalent to requiring that the points $P_1,\dots,P_6$ lie on a (possibly degenerate) conic (cf.,~\cite[p.118]{Cob61} and \cite[Example~3.4.3]{Stu08}).
In this way, one obtains a Grassmann--Cayley algebra proof of Pascal's Theorem.
\end{example}


\subsection{The equations for $d+4$ points on a rational normal curve in $\mathbb{P}^d$}
\label{equationsforpointsonarationalnormalcurve}

Let $n,d$ be positive integers such that $n\geq d+4$.
A \emph{rational normal curve} in $\mathbb{P}^d$ is a smooth rational curve of degree $d$. Up to projective isomorphism there is a unique rational normal curve in $\mathbb{P}^d$. So for example, for $d=2$ a rational normal curve is just a smooth conic, and for $d=3$ a twisted cubic. In this context, it is natural to consider the subvariety of $(\mathbb{P}^d)^n$ consisting of the $n$-tuples of distinct points in $\mathbb{P}^d$ that lie on a rational normal curve. This parameter space can be compactified by taking its Zariski closure in $(\mathbb{P}^d)^n$. The resulting projective variety is denoted by $V_{d,n}$, and is called the \emph{Veronese compactification}.

Since $V_{d,n}$ is defined as a Zariski closure, it is reasonable to expect that some of the point configurations parametrized by it, are supported on degenerations of rational normal curves. These degenerations are the so-called \emph{quasi-Veronese curves}, which are complete, connected, curves of degree $d$ in $\mathbb{P}^d$ not contained in a hyperplane. By a result of Artin, quasi-Veronese curves are built out of rational normal curves in the following way: each irreducible component of a quasi-Veronese curve is a rational normal curve in its span, each singularity of a quasi-Veronese curve is étale locally the union of coordinate axis, and finally each connected closed subcurve of a quasi-Veronese curve is again a quasi-Veronese curve in its span. For instance, the degree three quasi-Veronese curves are: twisted cubic, non-coplanar union of line and conic, chain of three non-coplanar lines, and non-coplanar union of three lines meeting at a point.

\par It is natural to ask what are the multi-homogeneous equations defining $V_{d,n}$, at least set-theoretically. For instance, for $d=2$ and $n=6$ the answer is given by the equation \eqref{eq:pascal2} of Example \ref{example:pascal}. Thus, $V_{2,6}$ is a hypersurface in $(\mathbb{P}^2)^6$. Moreover, by pulling back the previous equation along forgetful maps, one can obtain defining equations for $V_{2,n}$ for all $n\geq6$ (cf., \cite[Theorem 3.6]{CGMS18}).

In higher degree $d\geq3$, the story is more involved. 
We denote by $Y_{d,n}\subseteq(\mathbb{P}^d)^n$ the locus of $n$-point configurations which lie on a common hyperplane. $Y_{d,n}$ is a determinantal variety defined by all $(d+1)\times(d+1)$ minors of the $(d+1)\times n$ matrix whose columns are given by homogeneous coordinates of each copy of $\mathbb{P}^d$.
For the purpose of the current paper, we focus on the case $n=d+4$. 
Using the Gale transform, one can provide equations defining set-theoretically $V_{d,d+4}\cup Y_{d,d+4}$ (cf. \cite[Theorem 4.19]{CGMS18}). 
Since these will be useful later on, we briefly recall their construction.

\begin{notation}
	Let $m$ be a positive integer. In what follows, it is convenient to denote by $[m]$ the set $\{1,\ldots,m\}$. Caution: $[m]$ also denotes a bracket of length $1$. It will be clear from context which one of the two interpretations we mean. If $k\leq m$ is a positive integer, then let $\binom{[m]}{k}$ be the set of $k$-element subsets of $[m]$.
\end{notation}

\begin{definition}[{\cite[Definition~4.13]{CGMS18}}]
\label{definitionequationspsi}
Let $I\in\binom{[d+4]}{6}$, $I=\{i_1<\ldots<i_6\}$. Consider the equation in $(\mathbb{P}^2)^6$ given by
\[
\phi_I=[i_1i_2i_3][i_1i_4i_5][i_2i_4i_6][i_3i_5i_6]-[i_1i_2i_4][i_1i_3i_5][i_2i_3i_6][i_4i_5i_6]=0.
\]
Define $\psi_I$ to be the equation in $(\mathbb{P}^d)^{d+4}$ obtained from $\phi_I$ by operating the following substitution on the brackets:
\[
[i_\ell i_mi_n]\mapsto(-1)^{S_{\{i_\ell i_mi_n\}}}[\{i_\ell i_mi_n\}^c],
\]
where the complement $\{i_\ell i_mi_n\}^c$ is taken in $[d+4]$ and $S_{\{i_\ell i_mi_n\}}$ is the number of adjacent transpositions necessary to move the indices $i_\ell,i_m,i_n$ to $1,2,3$ respectively. Moreover, let $W_{d,d+4}\subseteq(\mathbb{P}^d)^{d+4}$  be the scheme defined by the equations
\[
\psi_I,~\textrm{for}~I\in\binom{[d+4]}{6}.
\]
\end{definition}

\begin{theorem}[{\cite[Theorem~4.19]{CGMS18}}]
$W_{d,d+4}=V_{d,d+4}\cup Y_{d,d+4}$ set-theoretically for all $d\geq3$.
\end{theorem}

We conclude this section with an explicit example of equation $\psi_I$. 
\begin{example}
Let $d=3$ and $I=[6]\in\binom{[7]}{6}$. We have that
\[
	\phi_{[6]} = |123||145||246||356| 
	- |124||135||236||456|.
\]
Therefore, by following the procedure in Definition~\ref{definitionequationspsi}, we obtain that
\[
	\psi_{[6]} = |4567||2367||1357||1247| - |3567||2467||1457||1237|.
\]
The other six equations defining $W_{3,7}$ can be found in a similar way for different choices of the index set $I\in\binom{[7]}{6}$.
\end{example}


\section{Lifting van der Waerden syzygies}
\label{liftvanderwaerdensyzygies}

The main goal of this section is to prove a technical lemma which allows us to produce van der Waerden syzygies in $\Bbbk[\Lambda(J',d')]$ starting from syzygies in $\Bbbk[\Lambda(J,d)]$, where $J\subseteq J'$ is a subset, and $d\leq d'$. 
\par Let $n,d$ be positive integers with $n\geq d$, and let $J\subsetneq\{1,\dots,n\}$.

\begin{definition}\label{def:lift}
Given $j\in\{1,\dots,n\}\setminus J$, we define a homomorphism of $\Bbbk$-algebras
\[
\eta_j\colon\Bbbk[\Lambda(J,d)]\rightarrow\Bbbk[\Lambda(J\cup\{j\},d+1)],
\]
obtained by extending $[\lambda_{i_1}\ldots\lambda_{i_d}]\mapsto[\lambda_{i_1}\ldots\lambda_{i_d}~j]$. Given a bracket polynomial $p\in\Bbbk[\Lambda(J,d)]$, we call $\eta_j(p)$ the \emph{lift} of $p$.
\end{definition}

\begin{lemma}\label{lemma:liftingsyzygies}
If $[[\alpha\beta^\bullet\gamma]]$ is a van der Waerden syzygy, then
\[
\eta_{j}([[\alpha\beta^\bullet\gamma]])=[[\alpha\eta_{j}(\beta)^\bullet\eta_{j}(\gamma)]].
\]
In particular, $\eta_{j}(\mathcal{I}_{J,d})\subseteq\mathcal{I}_{J\cup\{j\},d+1}$. That is, the lift of syzygies are syzygies.
\end{lemma}

\begin{proof}
The second statement follow from the first, since van der Waerden syzygies are a system of generators of $\mathcal{I}_{J,d}$.
So, we prove the first claim. Define $\delta=\eta_{j}(\beta)$, so that $\delta_{d+2}=j$. We have that
\[
[[\alpha\eta_{j}(\beta)^\bullet\eta_{j}(\gamma)]]=[[\alpha\delta^\bullet\eta_{j}(\gamma)]]=\sum_{\tau\in\Lambda(d+2,s)}\sign(\tau,\tau^*)[\alpha\delta_{\tau_1^*}\ldots\delta_{\tau_{d+2-s}^*}][\delta_{\tau_1}\ldots\delta_{\tau_{s}}\eta_j(\gamma)]
\]
\[
=\sum_{\substack{\tau\in\Lambda(d+2,s),\\d+2\in\tau^*}}\sign(\tau,\tau^*)[\alpha\delta_{\tau_1^*}\ldots\delta_{\tau_{d+2-s}^*}][\beta_{\tau_1}\ldots\beta_{\tau_{s}}\gamma~j].
\]
Fix $\tau$ as in the sum above. Observe that $\tau_{d+2-s}^*=d+2$, because $\tau_1^*<\ldots<\tau_{d+2-s}^*$ and $d+2\in\tau^*$. Hence $\delta_{\tau_{d+2-s}^*}=\delta_{d+2}=j$, implying that
\[
[[\alpha\eta_{j}(\beta)^\bullet\eta_{j}(\gamma)]]=\sum_{\substack{\tau\in\Lambda(d+2,s),\\d+2\in\tau^*}}\sign(\tau,\tau^*)[\alpha\beta_{\tau_1^*}\ldots\beta_{\tau_{d+1-s}^*}j][\beta_{\tau_1}\ldots\beta_{\tau_{s}}\gamma~j].
\]
To conclude, again let $\tau$ be as in the sum above. Let $\sigma=\tau$ viewed as an element of $\Lambda(d+1,s)$, so that $\sigma^*$ equals $\tau^*$ with the last entry $\tau_{d+2-s}^*=d+2$ removed. In particular, $\sign(\tau,\tau^*)=\sign(\sigma,\sigma^*)$. Hence we can conclude that
\[
[[\alpha\eta_{j}(\beta)^\bullet\eta_{j}(\gamma)]]=\sum_{\sigma\in\Lambda(d+1,s)}\sign(\sigma,\sigma^*)[\alpha\beta_{\sigma_1^*}\ldots\beta_{\sigma_{d+1-s}^*}j][\beta_{\sigma_1}\ldots\beta_{\sigma_{s}}\gamma~j]
\]
\[
=\sum_{\sigma\in\Lambda(d+1,s)}\sign(\sigma,\sigma^*)\eta_{j}([\alpha\beta_{\sigma_1^*}\ldots\beta_{\sigma_{d+1-s}^*}])\eta_{j}([\beta_{\sigma_1}\ldots\beta_{\sigma_{s}}\gamma])=\eta_{j}([[\alpha\beta^\bullet\gamma]]).\qedhere
\]
\end{proof}

\begin{corollary}\label{cor-liftingequations}
Let $J\subseteq\mathbb{N}$ be a finite index set and let $d\leq|J|$ be a positive integer. Then for any $j\in\mathbb{N}\setminus J$ we have a well-defined homomorphism of $\Bbbk$-algebras $\eta_j\colon\mathcal{B}_{J,d}\rightarrow\mathcal{B}_{J\cup\{j\},d+1}$.
\end{corollary}

\begin{example}
Consider again the van der Waerden syzygy of Example~\ref{firstexampleofvanderwaerdensyzygy}. Then by Lemma~\ref{lemma:liftingsyzygies} one has the van der Waerden syzygy
\[
[1467][1357]-[1367][1457]-[1347][1567]=0.
\]
\end{example}


\section{Alternative description of the equations $\psi_I$}
\label{alternativedescriptionequationswithlifts}

Let $d\geq3$ be an integer. In this section, we obtain a different description of the equations $\psi_I,I\in\binom{[d+4]}{6}$, which cut out set-theoretically the variety $W_{d,d+4}=V_{d,d+4}\cup Y_{d,d+4}$.

\begin{proposition}\label{prop:newequationsWdn}
	Let $d\geq3$ be an integer. For each $I\in\binom{[d+4]}{6}$, we set $I=\{i_1<\ldots<i_6\}$ and $I^c=\{j_1<\ldots<j_{d-2}\}$. 
	Then the variety $W_{d,d+4}$ is defined set-theoretically by the equations
	\begin{equation*}
	(\eta_{j_{d-2}}\circ\ldots\circ\eta_{j_1})(\phi_I)=0\ \text{for each }\ I\in\binom{[d+4]}{6},
	\end{equation*}
	where $\phi_I=[i_1i_2i_3][i_1i_4i_5][i_2i_4i_6][i_3i_5i_6]-[i_1i_2i_4][i_1i_3i_5][i_2i_3i_6][i_4i_5i_6]$ and $\eta_{j_i}$ denotes the lift homomorphism in Definition \ref{def:lift}.
\end{proposition}	
\begin{proof}
$W_{d,d+4}$ is defined set-theoretically by the equations $\psi_I$, which recall are obtained from $\phi_I$ by operating the following substitution on the brackets:
\[
[i_\ell i_mi_n]\mapsto(-1)^{S_{\{i_\ell i_mi_n\}}}[\{i_\ell i_mi_n\}^c],
\]
where $S_{\{i_\ell i_mi_n\}}$ counts the number of adjacent transpositions necessary to move the indices $i_\ell,i_m,i_n$ to $1,2,3$ respectively.
\par For each such $I$, we can rewrite the substituted brackets as follows:
\[
(-1)^{S_{\{i_\ell i_mi_n\}}}[\{i_\ell i_mi_n\}^c]=(-1)^{S_{\{i_\ell i_mi_n\}}}(-1)^{S_{\{i_ai_bi_c\}}}[i_ai_bi_cj_1\ldots j_{d-2}],
\]
where $I=\{i_\ell,i_m,i_n,i_a,i_b,i_c\}$, $\ell<m<n$, and $a<b<c$. Set  $S(I)=S_{\{i_1i_2i_3\}}+S_{\{i_4i_5i_6\}}+S_{\{i_1i_4i_5\}}+S_{\{i_2i_3i_6\}}+S_{\{i_2i_4i_6\}}+S_{\{i_1i_3i_5\}}+S_{\{i_3i_5i_6\}}+S_{\{i_1i_2i_4\}}$. 
Using the equalities above, we obtain that
\begin{align*}
\psi_I=(-1)^{S(I)}(&[i_4i_5i_6j_1\ldots j_{d-2}][i_2i_3i_6j_1\ldots j_{d-2}][i_1i_3i_5j_1\ldots j_{d-2}][i_1i_2i_4j_1\ldots j_{d-2}]\\
-&[i_3i_5i_6j_1\ldots j_{d-2}][i_2i_4i_6j_1\ldots j_{d-2}][i_1i_4i_5j_1\ldots j_{d-2}][i_1i_2i_3j_1\ldots j_{d-2}]).
\end{align*}
At this point it is easy to observe that
\[
\psi_I=(-1)^{S(I)+1}(\eta_{j_{d-2}}\circ\ldots\circ\eta_{j_1})(\phi_I).
\]
Since we are interested in the vanishing locus, we can ignore the sign $(-1)^{S(I)+1}$. So, the claim is proved.
\end{proof}	

\begin{remark}
Observe that the number $S(I)$ used in the proof of Proposition \ref{prop:newequationsWdn} has the same parity, namely it is always even.
The reason is that each $i_j$ appears an even number of times (four times) in the expression of $S(I)$.
\end{remark}


\section{Main results}
\label{sectionwithmainresult}

\subsection{Generalized Pascal's Theorem}

We are now ready to give a proof of Theorem~\ref{maintheoremintro2} in the Introduction, which we restate here for the reader's convenience.

\begin{theoremb}
	Let $d\geq3$, let $P_1,\ldots,P_{d+4}$ be points in $\mathbb{P}^d$ not on a hyperplane. Then the following are equivalent:
	\begin{enumerate}[(i)]
		\item $(P_1,\ldots,P_{d+4})\in V_{d,d+4}$ (equivalently, they lie on a quasi-Veronese curve);
		\item For every $I=\{i_1<\ldots<i_6\}\subseteq\{1,\dots,d+4\}$, $I^c=\{j_1<\ldots<j_{d-2}\}$ the following equality in the Grassmann--Cayley algebra holds:
		\begin{gather*}
		(P_{i_1}P_{i_2}\wedge P_{i_4}P_{i_5}P_{j_1}\ldots P_{j_{d-2}})\vee(P_{i_2}P_{i_3}\wedge P_{i_5}P_{i_6}P_{j_1}\ldots P_{j_{d-2}})\\
		\vee(P_{i_3}P_{i_4}\wedge P_{i_6}P_{i_1}P_{j_1}\ldots P_{j_{d-2}})\vee P_{j_1}\ldots P_{j_{d-2}}=0.
		\end{gather*}
	\end{enumerate}
\end{theoremb}

\begin{proof}
Recall from Proposition \ref{prop:newequationsWdn} that $(\eta_{j_{d-2}}\circ\ldots\circ\eta_{j_1})(\phi_I)=0$ for $I\in\binom{[d+4]}{6}$ are the defining equations of $W_{d,d+4}=V_{d,d+4}\cup Y_{d,d+4}$. 
Observe that since $P_1,\ldots,P_{d+4}$ are not on a hyperplane by assumption, then $(P_1,\ldots,P_{d+4})\notin Y_{d,d+4}$.
Therefore we have that $(P_1,\ldots,P_{d+4})\in V_{d,d+4}$ if and only if they satisfy the equations $(\eta_{j_{d-2}}\circ\ldots\circ\eta_{j_1})(\phi_I)=0$ for each $I\in\binom{[d+4]}{6}$.
\par Fix $I\in\binom{[d+4]}{6}$ and consider the corresponding Grassmann--Cayley algebra expression as in \textrm{(ii)}. We expand it in the bracket polynomial algebra $\Bbbk[\Lambda(I\cup I^c,d+1)]$ and show that is equivalent to the equation $(\eta_{j_{d-2}}\circ\ldots\circ\eta_{j_1})(\phi_I)=0$ modulo appropriate syzygies. 
This would prove what we need. 

\par We start by expanding the three meets. For instance, the first meet becomes
\[
P_{i_1}P_{i_2}\wedge P_{i_4}P_{i_5}P_{j_1}\ldots P_{j_{d-2}}=[P_{i_1}P_{i_4}P_{i_5}P_{j_1}\ldots P_{j_{d-2}}]P_{i_2}-[P_{i_2}P_{i_4}P_{i_5}P_{j_1}\ldots P_{j_{d-2}}]P_{i_1}.
\]
Let us denote $[P_{a}P_{b}P_{c}P_{j_1}\ldots P_{j_{d-2}}]$ simply by $[abcj_1\ldots j_{d-2}]$. After distributing the joins with respect to the sums, we obtain the simplified expression
\begin{equation}\label{eq:Vdd+41}
\begin{split}
&[i_1i_4i_5j_1\ldots j_{d-2}][i_2i_5i_6j_1\ldots j_{d-2}][i_3i_6i_1j_1\ldots j_{d-2}][i_2i_3i_4j_1\ldots j_{d-2}]+\\
- &[i_2i_4i_5j_1\ldots j_{d-2}][i_3i_5i_6j_1\ldots j_{d-2}][i_4i_6i_1j_1\ldots j_{d-2}][i_1i_2i_3j_1\ldots j_{d-2}]+\\
+ &[i_2i_4i_5j_1\ldots j_{d-2}][i_3i_5i_6j_1\ldots j_{d-2}][i_3i_6i_1j_1\ldots j_{d-2}][i_1i_2i_4j_1\ldots j_{d-2}]+\\
- &[i_2i_4i_5j_1\ldots j_{d-2}][i_2i_5i_6j_1\ldots j_{d-2}][i_3i_6i_1j_1\ldots j_{d-2}][i_1i_3i_4j_1\ldots j_{d-2}] = 0.
\end{split}
\end{equation}
Observe that this bracket polynomial is obtained by applying consecutive lifts $\eta_{j_1},\dots,\eta_{j_{d-2}}$ to the following equation in $\Bbbk[\Lambda(I,3)]$
\begin{equation*}
\begin{split}
&[i_1i_4i_5][i_2i_5i_6][i_3i_6i_1][i_2i_3i_4]-[i_2i_4i_5][i_3i_5i_6][i_4i_6i_1][i_1i_2i_3]+\\
+&[i_2i_4i_5][i_3i_5i_6][i_3i_6i_1][i_1i_2i_4]-[i_2i_4i_5][i_2i_5i_6][i_3i_6i_1][i_1i_3i_4]=0.
\end{split}
\end{equation*}
As we did in Example~\ref{example:pascal}, applying the straightening algorithm to this equation, yields the unique standard representation in $\mathcal{B}_{I,3}$
\begin{equation*} 
[i_1i_2i_3][i_1i_4i_5][i_2i_4i_6][i_3i_5i_6]-[i_1i_2i_4][i_1i_3i_5][i_2i_3i_6][i_4i_5i_6]=0.
\end{equation*}
Now, applying lifts $\eta_{j_1},\dots,\eta_{j_{d-2}}$ to the previous equation, we obtain the following equation in $\mathcal{B}_{I\cup I^c,d+1}$, which is equivalent to equation~\eqref{eq:Vdd+41} by Corollary \ref{cor-liftingequations}:
\begin{equation}\label{eq:Vdd+42}
\begin{split}
&[i_1i_2i_3j_1\ldots j_{d-2}][i_1i_4i_5j_1\ldots j_{d-2}][i_2i_4i_6j_1\ldots j_{d-2}][i_3i_5i_6j_1\ldots j_{d-2}]+\\
- &[i_1i_2i_4j_1\ldots j_{d-2}][i_1i_3i_5j_1\ldots j_{d-2}][i_2i_3i_6j_1\ldots j_{d-2}][i_4i_5i_6j_1\ldots j_{d-2}]=0.
\end{split}
\end{equation}
Finally, observe that equation~\eqref{eq:Vdd+42} is exactly $(\eta_{j_{d-2}}\circ\ldots\circ\eta_{j_1})(\phi_I)=0$, which is one of the defining equations of $W_{d,d+4}$. Repeating the previous reasoning for all $I\in\binom{[d+4]}{6}$, we obtain all the defining equations of $W_{d,d+4}$.
\end{proof}

We illustrate the central step of the proof of Theorem~\ref{maintheoremintro2} in the following example.

\begin{example}\label{ex:V37}
	Consider the following bracket equation in $\Bbbk[\Lambda(7,4)]$, obtained by applying $\eta_7$ to \eqref{eq:pascal1}
	\begin{equation}\label{eq:V371}
	\begin{split}
	&[1457][2567][3617][2347]-[2457][3567][4617][1237]+\\
	+&[2457][3567][3617][1247]-[2457][2567][3617][1347]=0.
	\end{split}
	\end{equation}
	Since by Lemma \ref{lemma:liftingsyzygies} we have  $\eta_{7}(\mathcal{I}_{6,3})\subseteq\mathcal{I}_{7,4}$, we know that the lift of the syzygies \eqref{eq:syzygiespascal} are syzygies for the bracket algebra $\Bbbk[\Lambda(7,4)]$.
	Therefore, applying the straightening algorithm to \eqref{eq:V371} yields the unique standard bracket representation
	\begin{equation*}
	[1237][1457][2467][3567]-[1247][1357][2367][4567]=0,
	\end{equation*}
	which is also obtained by applying $\eta_7$ to \eqref{eq:pascal2}, the unique standard representation  of \eqref{eq:pascal1}.
\end{example}

The following corollary follows immediately from the geometric interpretation of the Grassmann--Cayley algebra expression in Theorem~\ref{maintheoremintro2}~(ii).

\begin{corollary}[Generalized Pascal's Theorem]\label{corollary-main}
Let $P_1,\ldots,P_{d+4}$ be points in $\mathbb{P}^d$ not on a hyperplane. Then $P_1,\ldots,P_{d+4}$ lie on a quasi-Veronese curve if and only if for every $I\in\binom{[d+4]}{6}$, $I=\{i_1<\ldots<i_6\}$, $I^c=\{j_1<\ldots<j_{d-2}\}$, the following $d+1$ points lie on a hyperplane:
\begin{itemize}
\item The intersection of the line $P_{i_1}P_{i_2}$ with the hyperplane $P_{i_4}P_{i_5}P_{j_1}\ldots P_{j_{d-2}}$;
\item The intersection of the line $P_{i_2}P_{i_3}$ with the hyperplane $P_{i_5}P_{i_6}P_{j_1}\ldots P_{j_{d-2}}$;
\item The intersection of the line $P_{i_3}P_{i_4}$ with the hyperplane $P_{i_6}P_{i_1}P_{j_1}\ldots P_{j_{d-2}}$;
\item The points $P_{j_1},\ldots,P_{j_{d-2}}$.
\end{itemize}
In particular, if $P_1,\ldots,P_{d+4}$ are in general linear position then the previous conditions are equivalent to requiring that $P_1,\ldots,P_{d+4}$ lie on a rational normal curve.
\end{corollary}

In the three-dimensional case, that is for seven points in $\mathbb{P}^3$ the situation is particularly nice.
The fact that seven points $P_1,\dots,P_7$ lie on a twisted cubic implies, by choosing $I=\{1,\dots,6\}$ in the previous corollary, that the three intersection points 
$\overline{P_{1}P_{2}}\cap \overline{P_{4}P_{5}P_{7}}$, $\overline{P_{2}P_{3}}\cap \overline{P_{5}P_{6}P_{7}}$, and $\overline{P_{3}P_{4}}\cap \overline{P_{6}P_{1}P_{7}}$ are coplanar with $P_7$. We illustrate this in Figure \ref{picture7pointswithtwistedcubic}.

\begin{figure}
	\centering
	\includegraphics[scale=0.30,valign=t]{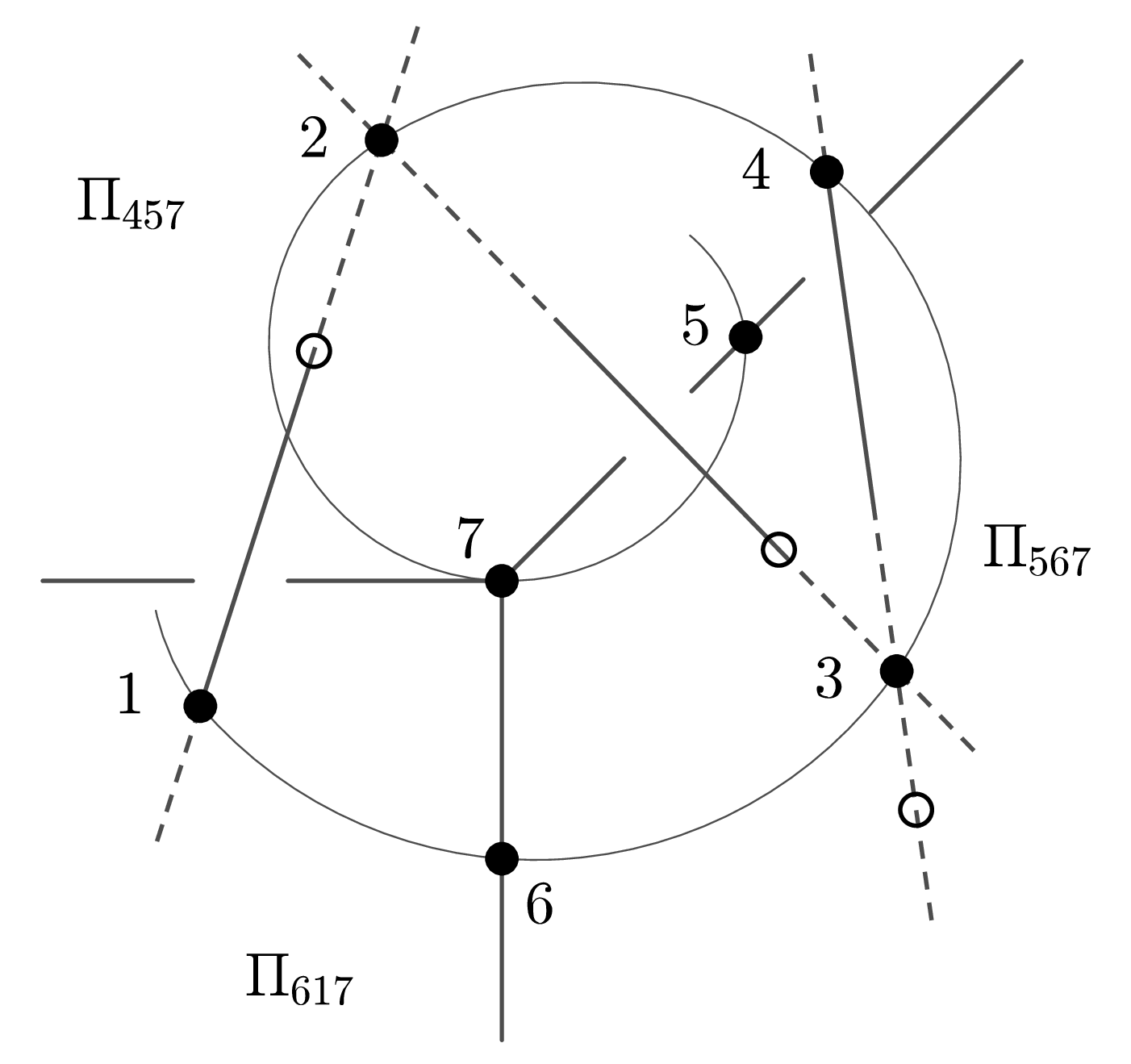}
	\caption{Seven points in $\mathbb{P}^3$ lying on a twisted cubic. By Corollary~\ref{corollary-main}, the three circled points and $P_7$ are coplanar. $\prod_{ijk}$ denotes the plane containing $P_i,P_j,P_k$. A line changes from continuous to dotted when it crosses one of these planes.}
	\label{picture7pointswithtwistedcubic}
\end{figure}

\begin{remark}
For $d=3,4$, in Corollary \ref{corollary-main} one can remove the assumption that the starting points do not lie on a common hyperplane. This follows from the fact that by \cite[Proposition 4.25]{CGMS18} the varieties $Y_{3,7}$ and $Y_{4,8}$ parametrizing point configurations supported on a hyperplane are contained in $V_{3,7}$ and $V_{4,8}$ respectively.
\end{remark}

\subsection{Equivalent formulations}

The Grassmann--Cayley algebra equation in Theorem~\ref{maintheoremintro2}~(ii) can be rewritten in many equivalent ways using the standard properties of the meet and join operations. This is the content of the next theorem.

\begin{theorem}
\label{equivalentformulations}
Let $\{j_1,\ldots,j_{d-2}\}=J_1\amalg J_2\amalg J_3\amalg J_4$ be a partition, where in each $J_i$ the indices are in ascending order ($J_i$ could possibly be empty). Then the Grassmann--Cayley algebra equation in Theorem~\ref{maintheoremintro2}~(ii) can be rewritten as
\begin{gather*}
(P_{i_1}P_{i_2}(\vee_{j\in J_1}P_j)\wedge P_{i_4}P_{i_5}P_{j_1}\ldots P_{j_{d-2}})\vee(P_{i_2}P_{i_3}(\vee_{j\in J_2}P_j)\wedge P_{i_5}P_{i_6}P_{j_1}\ldots P_{j_{d-2}})\\
\vee(P_{i_3}P_{i_4}(\vee_{j\in J_3}P_j)\wedge P_{i_6}P_{i_1}P_{j_1}\ldots P_{j_{d-2}})\vee(\vee_{j\in J_4}P_j)=0.
\end{gather*}
\end{theorem}

\begin{proof}
We start with the equation above, and we show that it is equivalent to the equation in Theorem~\ref{maintheoremintro2}~(ii).
We expand the first meet using \eqref{eq-meet}
\begin{equation*}
\begin{split}
P_{i_1}P_{i_2}(\vee_{j\in J_1}P_j)\wedge P_{i_4}P_{i_5}P_{j_1}\ldots P_{j_{d-2}} = 
& +[P_{i_1}P_{i_4}P_{i_5}P_{j_1}\ldots P_{j_{d-2}}]P_{i_2}(\vee_{j\in J_1}P_j)  \\
& -[P_{i_2}P_{i_4}P_{i_5}P_{j_1}\ldots P_{j_{d-2}}]P_{i_1}(\vee_{j\in J_1}P_j).
\end{split}
\end{equation*}
Observe that we have only two non-zero summands in the previous expansion, since each bracket of the form $[P_jP_{i_4}P_{i_5}P_{j_1}\ldots P_{j_{d-2}}]=0$ for $j\in J_1$, being $J_1$ a subset of $\{j_1,\ldots,j_{d-2}\}$.
Thus, collecting $(\vee_{j\in J_1}P_j)$ and writing back in the Grassmann--Cayley algebra, yields the equality 
\begin{equation*}
P_{i_1}P_{i_2}(\vee_{j\in J_1}P_j)\wedge P_{i_4}P_{i_5}P_{j_1}\ldots P_{j_{d-2}} = (P_{i_1}P_{i_2}\wedge P_{i_4}P_{i_5}P_{j_1}\ldots P_{j_{d-2}})\vee(\vee_{j\in J_1}P_j).
\end{equation*}
Repeating the same reasoning for the other two meets and rearranging the joins, we obtain that the expression in the statement of the theorem is equal to
\begin{gather*}
(P_{i_1}P_{i_2}\wedge P_{i_4}P_{i_5}P_{j_1}\ldots P_{j_{d-2}})\vee(P_{i_2}P_{i_3}\wedge P_{i_5}P_{i_6}P_{j_1}\ldots P_{j_{d-2}})\\
\vee(P_{i_3}P_{i_4}\wedge P_{i_6}P_{i_1}P_{j_1}\ldots P_{j_{d-2}})\vee (\vee_{j\in J_1}P_j)\vee(\vee_{j\in J_2}P_j)\vee(\vee_{j\in J_3}P_j)\vee(\vee_{j\in J_4}P_j)=0,
\end{gather*}
which is the equation in Theorem~\ref{maintheoremintro2}~(ii) since $J_1\amalg J_2\amalg J_3\amalg J_4=\{j_1,\dots,j_{d-2}\}$.
\end{proof}

\par Each one of the Grassmann--Cayley algebra equations in Theorem~\ref{equivalentformulations} leads to a distinct, yet equivalent, reformulation of the geometric statement of Corollary~\ref{corollary-main}. Let us look at a specific example for $d=3$. 
\begin{example}
For each $I=\{i_1<\ldots<i_6\}\in\binom{[7]}{6}$, by Theorem~\ref{equivalentformulations} we can rewrite the equation in Theorem~\ref{maintheoremintro2}~(ii) as
\[
(P_{i_1}P_{i_2}P_{j_1}\wedge P_{i_4}P_{i_5}P_{j_1})\vee(P_{i_2}P_{i_3}\wedge P_{i_5}P_{i_6}P_{j_1})\vee(P_{i_3}P_{i_4}\wedge P_{i_6}P_{i_1}P_{j_1})=0.
\]
(Note that for different $I\in\binom{[7]}{6}$ we may choose different ways of partitioning the set $\{j_1\}$ as $J_1\amalg J_2\amalg J_3\amalg J_4$. Here we always choose $J_1=\{j_1\}$ and $J_2=J_3=J_4=\emptyset$ for the sake of example.) Therefore, an equivalent formulation of Corollary~\ref{corollary-main} for $d=3$ is the following. \emph{Let $P_1,\ldots,P_7$ be points in $\mathbb{P}^3$ not on a hyperplane. Then $P_1,\ldots,P_7$ lie on a quasi-Veronese curve if and only if for every $I\in\binom{[7]}{6}$, $I=\{i_1<\ldots<i_6\}$, $I^c=\{j_1\}$, the following $3$ linear subspaces of $\mathbb{P}^3$ lie on a hyperplane:
\begin{itemize}
\item The line of intersection of the planes $P_{i_1}P_{i_2}P_{j_1}$ and $P_{i_4}P_{i_5}P_{j_1}$;
\item The point of intersection of the line $P_{i_2}P_{i_3}$ with the plane $P_{i_5}P_{i_6}P_{j_1}$;
\item The point of intersection of the line $P_{i_3}P_{i_4}$ with the plane $P_{i_6}P_{i_1}P_{j_1}$.
\end{itemize}
}
\end{example}

\subsection{Generalized Brianchon's Theorem}

\par The projective dual of Pascal's Theorem is known as Brianchon's Theorem: If six distinct lines are tangent to a smooth conic, then the three lines joining opposite vertices of the hexagon are concurrent \cite[Chapter XIV]{Cre60}. By dualizing one of the implications of the geometric statement in Corollary~\ref{corollary-main}, we obtain a generalization of Brianchon's Theorem to rational normal curves in $\mathbb{P}^d$, $d\geq3$. More precisely, if the characteristic of our base field $\Bbbk$ is zero or greater than $d$, then by \cite[\S5]{Pie77} (see also \cite[Chapter~III, Exercise~A-2]{ACGH85} in characteristic zero) we have that the set of osculating hyperplanes to a rational normal curve, viewed as points in the dual $(\mathbb{P}^d)^*$, is a rational normal curve in $(\mathbb{P}^d)^*$. We can then state the following result.

\begin{corollary}[Generalized Brianchon's Theorem]
\label{generalizedbrianchon'stheorem}
Let $\Bbbk$ be an algebraically closed field with $\chara\Bbbk=0$ or $\chara\Bbbk>d$, and let $H_1,\ldots,H_{d+4}$ be hyperplanes in $\mathbb{P}^d$ in general linear position which osculate a rational normal curve. Then for every $I\in\binom{[d+4]}{6}$, $I=\{i_1<\ldots<i_6\}$, $I^c=\{j_1<\ldots<j_{d-2}\}$, the following $d+1$ hyperplanes have nonempty intersection:
	\begin{itemize}
		\item The linear span of $H_{i_1}\cap H_{i_2}$ with the point $H_{i_4}\cap H_{i_5}\cap H_{j_1}\cap \ldots\cap H_{j_{d-2}}$;
		\item The linear span of $H_{i_2}\cap H_{i_3}$ with the point $H_{i_5}\cap H_{i_6}\cap H_{j_1}\cap\ldots\cap H_{j_{d-2}}$;
		\item The linear span of $H_{i_3}\cap H_{i_4}$ with the point $H_{i_6}\cap H_{i_1}\cap H_{j_1}\cap\ldots\cap H_{j_{d-2}}$;
		\item The hyperplanes $H_{j_1},\ldots,H_{j_{d-2}}$.
	\end{itemize}
\end{corollary}


\section{Application to seven points on a twisted cubic}
\label{applicationofresults}

For simplicity, in this section we work over $\mathbb{C}$. The case of seven points on a twisted cubic in $\mathbb{P}^3$ is of great interest: in 1915 H.~White proved the following result.

\begin{theorem}{\cite{Whi15}}
\label{white'sresult}
Fix seven distinct points on a twisted cubic. Let $H_1,\ldots,H_7$ be seven planes whose union contains the $21$ lines spanned by the seven points (each one of these planes has to contain exactly three of the initial points). Then $H_1,\ldots,H_7$ osculate a second twisted cubic.
\end{theorem}

\begin{remark}
\label{geometryofwhite'ssituation}
As White discussed in \cite{Whi15}, the geometry involved in Theorem~\ref{white'sresult} is quite rich. For instance, label by $a,b,c,d,e,f,g$ the seven fixed points on the twisted cubic. Let $X$ be the set consisting of these points. Then each one of the planes $H_1,\ldots,H_7$ has to contain exactly three of the points in $X$. The collection of these $3$-elements subsets of $X$ forms a Steiner's system $S(2,3,7)$, which is the Fano plane $\mathbb{P}_{\mathbb{F}_2}^2$. An example of such system on $X$ is
\[
\{\{a,d,e\},\{a,f,g\},\{b,d,f\},\{b,e,g\},\{c,d,g\},\{c,e,f\},\{a,b,c\}\},
\]
and this can be determined in $|S_7|/|\PGL(3,\mathbb{F}_2)|=30$ different ways. Therefore, the planes $H_1,\ldots,H_7$ can be chosen in $30$ distinct ways, up to relabeling them. Finally, observe that the planes $H_1,\ldots,H_7$ are in general linear position. To prove this, first notice that $H_1,\ldots,H_7$ osculate a second twisted cubic by Theorem~\ref{white'sresult}. Therefore, by \cite[Chapter~III, Exercise~A-2]{ACGH85}, the points in $(\mathbb{P}^3)^*$ dual to $H_1,\ldots,H_7$ lie on a twisted cubic. Since distinct points on a twisted cubic are in general linear position, we have that also $H_1,\ldots,H_7$ are in general linear position.
\end{remark}

The combination of Theorem~\ref{maintheoremintro2}, Theorem~\ref{white'sresult}, and projective duality yields the following property of the planes $H_1,\ldots,H_7$.

\begin{theorem}
	With the setup of Theorem~\ref{white'sresult}, let $I=\{i_1<\ldots<i_6\}\subseteq[7]$ and $I^c=\{j\}$. Then the intersection of the following three planes is a point, and it lies on $H_j$:
	\begin{itemize}
		\item The linear span of $H_{i_1}\cap H_{i_2}$ with the point $H_{i_4}\cap H_{i_5}\cap H_{j}$;
		\item The linear span of $H_{i_2}\cap H_{i_3}$ with the point $H_{i_5}\cap H_{i_6}\cap H_{j}$;
		\item The linear span of $H_{i_3}\cap H_{i_4}$ with the point $H_{i_6}\cap H_{i_1}\cap H_{j}$.
	\end{itemize}
\end{theorem}

\begin{proof}
We adopt the following notations. A plane $H_i$ is simply denoted by its subscript $i$. Moreover, if we want to think of the plane $i$ as a point in the dual projective space $(\mathbb{P}^3)^*$, then we denote it by $i^*$. Observe that, by the discussion in Remark~\ref{geometryofwhite'ssituation}, the planes $1,\ldots,7$ are in general linear position (hence, also the points $1^*,\ldots,7^*$ are).

Let us first prove that the intersection of the three planes is a point. Assume by contradiction this is not the case. Then, in the Grassmann--Cayley algebra of $\mathbb{P}^3$, we must have that
\[
((i_1\wedge i_2)\vee(i_4\wedge i_5\wedge j))\wedge((i_2\wedge i_3)\vee(i_5\wedge i_6\wedge j))\wedge((i_3\wedge i_4)\vee(i_6\wedge i_1\wedge j))=0.
\]
Dually, in the Grassmann--Cayley algebra of $(\mathbb{P}^3)^*$ we have that
\begin{equation}
\label{threepointcollinear}
(i_1^*i_2^*\wedge i_4^*i_5^*j^*)\vee(i_2^*i_3^*\wedge i_5^*i_6^*j^*)\vee(i_3^*i_4^*\wedge i_6^*i_1^*j^*)=0,
\end{equation}
which says that the three points in parentheses are collinear. Observe that the point $i_1^*i_2^*\wedge i_4^*i_5^*j^*$ on the line $i_1^*i_2^*$ is different from $i_1^*$ and $i_2^*$ because $i_1^*,i_2^*,i_4^*,i_5^*,j^*$ are in general linear position. A similar argument applies to $(i_2^*i_3^*\wedge i_5^*i_6^*j^*)$ and $(i_3^*i_4^*\wedge i_6^*i_1^*j^*)$. But then, equation~\eqref{threepointcollinear} implies that the points $i_1^*,i_2^*,i_3^*,i_4^*$ are coplanar, which is a contradiction.

Let us prove that the intersection point lies on $j$. By Theorem~\ref{white'sresult}, the planes $1,\ldots,7$ osculate a twisted cubic. Hence, by projective duality, the points $1^*,\ldots,7^*$ lie on a twisted cubic. Therefore, by Theorem~\ref{maintheoremintro2} we have that the following expression in the Grassmann--Cayley algebra of $(\mathbb{P}^3)^*$ is satisfied:
\[
(i_1^*i_2^*\wedge i_4^*i_5^*j^*)\vee(i_2^*i_3^*\wedge i_5^*i_6^*j^*)\vee(i_3^*i_4^*\wedge i_6^*i_1^*j^*)\vee j^*=0.
\]
Dually, in the projective space $\mathbb{P}^3$ we have that
\[
((i_1\wedge i_2)\vee(i_4\wedge i_5\wedge j))\wedge((i_2\wedge i_3)\vee(i_5\wedge i_6\wedge j))\wedge((i_3\wedge i_4)\vee(i_6\wedge i_1\wedge j))\wedge j=0.
\]
Using the geometric interpretation of the above Grassmann--Cayley algebra expression, we have our claim.
\end{proof}




\begin{thebibliography}{ACGH85}

\bibitem[ACGH85]{ACGH85} {\sc Enrico Arbarello, Maurizio Cornalba, Phillip A. Griffiths, Joseph Harris}, \textit{Geometry of Algebraic Curves. Volume I}, Grundlehren der matematischen Wissenschaften, vol. 267, Springer, 1985.

\bibitem[Bac86]{Bac86} {\sc Isaak Bacharach}, {\em \"Uber den Cayley'schen Schrittpunktsatz}, Math. Ann. 26, no. 2, pp. 275--299, 1886.

\bibitem[BBR85]{BBR85} {\sc Marilena Barnabei, Andrea Brini, Gian-Carlo Rota}, {\em On the exterior calculus of invariant theory}, J. Algebra 96, no. 1, pp. 120--160, 1985.

\bibitem[B\.Z02]{BZ02} {\sc Maciej Borodzik, Henryk \.Zo\l\k adek}, {\em The Pascal theorem and some its generalizations}, Topol. Methods Nonlinear Anal. 19, no. 1, pp. 77--90, 2002.

\bibitem[Bra33]{Bra33} {\sc William Braikenridge}, {\em Exercitatio Geometrica de Descriptione Curvarum}, London, 1733.

\bibitem[CGMS18]{CGMS18} {\sc Alessio Caminata, Noah Giansiracusa, Han-Bom Moon, Luca Schaffler}, {\em Equations for point configurations to lie on a rational normal curve}, Adv. Math. 340, pp. 653--683, 2018.

\bibitem[CGMS20]{CGMS20} {\sc Alessio Caminata, Noah Giansiracusa, Han-Bom Moon, Luca Schaffler}, {\em Point configurations, phylogenetic trees, and dissimilarity vectors}, Proceedings of the National Academy of Sciences of the United States of America (PNAS) 118, n. 12, 2021.

\bibitem[Cay43]{Cay43} {\sc Arthur Cayley}, {\em On the intersection of curves}, Cambridge Math. J. 3 (1843), pp. 211--213; {\em Collected math papers} I, vols. 25--27, Cambridge Univ. Press, Cambridge, 1889.

\bibitem[Cha85]{Cha85} {\sc Michel Chasles}, {\em Trait\'e des sections coniques}, Gauthier-Villars, Paris, 1885.

\bibitem[Cob61]{Cob61} {\sc Arthur B. Coble}, {\em Algebraic geometry and theta functions}, Revised printing. American Mathematical Society Colloquium Publication, vol. X, American Mathematical Society, Providence, R.I. 1961 vii+282 pp.

\bibitem[Cre60]{Cre60}  {\sc Luigi Cremona}, {\em Elements of Projective Geometry. 3rd ed. Translated by Charles Leudesdorf}, Dover Publications, Inc., New York 1960 xx+302 pp.

\bibitem[EGH96]{EGH96} {\sc David Eisenbud, Mark Green, Joe Harris}, {\em Cayley-Bacharach theorems and conjectures}, Bull. Amer. Math. Soc. (N.S.) 33, no. 3, pp. 295--324, 1996.

\bibitem[Jam30]{Jam30} {\sc Glenn James}, {\em Generalizations of Pascal's and Brianchon's Theorems}, { Amer. Math. Monthly} 37, no. 2, pp. 78--80, 1930.

\bibitem[Mac35]{Mac35} {\sc Colin MacLaurin}, {\em A Letter from Mr. Colin Mac Laurin, Math. Prof. Edinburg. F.R.S. to Mr. John Machin, Ast. Prof. Gresh. \& Secr. R.S. concerning the Description of Curve Lines}, Phil. Trans. 39, pp. 143--165, 1735--36.

\bibitem[M\"ob48]{Mob48} {\sc August F. M\"obius}, \newblock {\em Verallgemeinerung des Pascalschen Theorems, das in ein Kegelschnitt beschriebene Sechseck betreffend}, {J. Reine Angew. Math.} 36, pp. 216--220, 1848.

\bibitem[Pas40]{Pas40} {\sc Blaise Pascal}, {\em Essay pour les coniques}, {Nieders\"achsiche Landesbibliothek, Gottfried Wilhelm Leibniz Bibliothek}, 1640.

\bibitem[Pie77]{Pie77}{\sc Ragni Piene}, {\em Numerical characters of a curve in projective n-space}, Real and complex singularities (Proc. Ninth Nordic Summer School/NAVF Sympos. Math., Oslo, 1976), pp. 475--495. Sijthoff	and Noordhoff, Alphen aan den Rijn, 1977.

\bibitem[Seg45]{Seg45} {\sc Beniamino Segre}, {\em A four-dimensional analogue of Pascal's theorem for conics}, {Amer. Math. Monthly} 52, pp. 119--131, 1945.

\bibitem[ST19]{ST19} {\sc Jessica Sidman, Will Traves}, \newblock {\em Special positions of frameworks and the Grassmann--Cayley algebra}, Handbook of Geometric Constraint Systems Principles, 85--106,
Discrete Math. Appl. (Boca Raton), CRC Press, Boca Raton, FL, 2019.

\bibitem[Stu08]{Stu08} {\sc Bernd Sturmfels}, \newblock {\em Algorithms in Invariant Theory. Second Edition}, in {\em Texts and Monographs in Symbolic Computation}, Springer-Verlag, Vienna, 2008.

\bibitem[SW89]{SW89} {\sc Bernd Sturmfels, Neil White}, \newblock {\em Gr\"obner bases and invariant theory}, Adv. Math. 76, pp. 245--259, 1989.

\bibitem[SW91]{SW91} {\sc Bernd Sturmfels, Walter Whiteley}, \newblock {\em On the synthetic factorization of projectively invariant polynomials}, J. Symbolic Comput. 11, pp. 439--454, 1991.

\bibitem[Tra13]{Tra13} {\sc Will Traves}, \newblock {\em From Pascal's theorem to $d$-constructible curves}, { Amer. Math. Monthly} 120, no. 10, pp. 901--915, 2013.

\bibitem[Whi15]{Whi15} {\sc Henry S. White}, \newblock {\em Seven points on a twisted cubic curve}, Proc. Natl. Acad. Sci. USA, Vol. 1, No. 8 , pp. 464--466, 1915.

\bibitem[Whi91]{Whi91} {\sc Neil L. White}, \newblock {\em Multilinear Cayley factorization},  J. Symbolic Comput. 11, pp. 421--438, 1991.
\end{thebibliography}
\end{document}